\documentclass[preprint,11pt]{elsarticle}
\makeatletter
\def\ps@pprintTitle{%
 \let\@oddhead\@empty
 \let\@evenhead\@empty
 \def\@oddfoot{\centerline{\thepage}}%
 \let\@evenfoot\@oddfoot}
\makeatother
\usepackage{amssymb,amsmath,amsthm}
\usepackage{epsfig}
\usepackage{color}
\usepackage{makeidx}
\usepackage{bbm}
\usepackage[top=1in, bottom=1in, inner=1in, outer=1in]{geometry}
\usepackage{setspace}

\newtheorem{theorem}{Theorem}
\newtheorem*{thma}{Theorem A}
\newtheorem*{thmb}{Theorem B}
\newtheorem{definition}{Definition}
\newtheorem{lemma}{Lemma}

\newtheorem{proposition}{Proposition}
\newtheorem{corollary}{Corollary}
\newtheorem*{remark}{Remark}
\newcommand*\xbar[1]{%
  \hbox{%
    \vbox{%
      \hrule height 0.5pt 
      \kern0.4ex
      \hbox{%
        \kern-0.15em
        \ensuremath{#1}%
        \kern-0.15em
      }%
    }%
  }%
}

\begin{document}
\begin{frontmatter}

\title{Large deviations for local mass of branching Brownian motion}

\author{Mehmet \"{O}z}
\ead{mehmet.oz@ozyegin.edu.tr}
\ead[url]{https://faculty.ozyegin.edu.tr/mehmetoz/}

\address{Department of Natural and Mathematical Sciences, Faculty of Engineering, \"{O}zye\u{g}in University, Istanbul, Turkey}

\begin{abstract}
We study the local mass of a dyadic branching Brownian motion $Z$ evolving in $\mathbb{R}^d$. By `local mass,' we refer to the number of particles of $Z$ that fall inside a ball with fixed radius and time-dependent center, lying in the `subcritical' zone. Using the strong law of large numbers for the local mass of branching Brownian motion and elementary geometric arguments, we find large deviation results giving the asymptotic behavior of the probability that the local mass is atypically small on an exponential scale. As corollaries, we obtain an asymptotic result for the probability of absence of $Z$ in a ball with fixed radius and time-dependent center, and lower tail asymptotics for the local mass in a fixed ball. The proofs are based on a bootstrap argument, which we use to find the lower tail asymptotics for the mass outside a ball with time-dependent radius and fixed center, as well.
\end{abstract}

\vspace{3mm}

\begin{keyword}
Branching Brownian motion \sep Large deviations \sep Local mass
\vspace{3mm}
\MSC[2010] 60J80 \sep 60F10 \sep 92D25
\end{keyword}

\end{frontmatter}

\pagestyle{myheadings}
\markright{Local mass of BBM\hfill}

\section{Introduction}\label{intro}

The setting in this paper is a branching Brownian motion (BBM) evolving in $\mathbb{R}^d$. A classical problem in this setting is how the local mass of BBM grows in time asymptotically as time tends to infinity. In this work, by `local mass,' we refer to the number of particles that fall inside a ball of fixed size with a possibly time-dependent center. By an elementary calculation based on a first moment formula, one can find the expected local mass at a given time. The strong law of large numbers for local mass of BBM in $\mathbb{R}^d$ was proved by Watanabe \cite{W1967}, and later improved by Biggins \cite{B1992}, saying that almost surely the local mass at time $t$ behaves as its expectation as $t$ tends to infinity. Hence, we know how the local mass typically grows for large time. In this paper, we mainly study the large deviations for local mass in the downward direction, and obtain large-time asymptotic results on the probability that the local mass is atypically small on an exponential scale (Theorem~\ref{theorem1}). We then study the lower tail asymptotics for the mass that falls outside a ball with a fixed center and time-dependent radius (Theorem~\ref{theorem2}). The proofs are based on a bootstrap argument, which is given in two stages, where the first stage is completed by Lemma~\ref{lemma3}. We now present the formulation of the problem and a brief review of history on the topic, followed by our motivation to study the current problem.   

\subsection{Formulation of the problem}
Let $Z=(Z(t))_{t\geq 0}$ be a $d$-dimensional strictly dyadic BBM with constant branching rate $\beta>0$. Here, $t$ represents time, and strictly dyadic means that every time a particle branches, it gives exactly two offspring. The process starts with a single particle, which performs a Brownian motion in $\mathbb{R}^d$ for a random exponential time of parameter $\beta>0$. Then, the particle dies and simultaneously gives birth to precisely two offspring. Similarly, starting from the position where their parent dies, each offspring particle repeats the same procedure as their parent independently of others and of the parent, and the process evolves through time in this way. The Brownian motions and exponential lifetimes of particles are all independent from one another. For each $t\geq 0$, $Z(t)$ can be viewed as a discrete measure on $\mathbb{R}^d$. Let $P_x$ and $E_x$, respectively, denote the probability and corresponding expectation for $Z$ when the process starts with a single particle at position $x\in\mathbb{R}^d$, that is, when $Z(0)=\delta_x$, denoting the Dirac measure at $x$. When $Z(0)=\delta_0$, we simply use $P$ and $E$. For a Borel set $B\subseteq \mathbb{R}^d$ and $t\geq 0$, we write $Z_t(B)$ to denote the number of particles, i.e., the \textit{mass}, of $Z$ that fall inside $B$ at time $t$, and we write $N_t:=Z_t(\mathbb{R}^d)$ for the total mass at time $t$. 

For $x\in\mathbb{R}^d$, we use $|x|$ to denote its Euclidean norm, and $B(x,r)$ to denote the open ball of radius $r>0$ centered at $x$. Also, for a Borel set $B$ and $x\in\mathbb{R}^d$, we define their sum in the sense of sum of sets as $B+x:=\{y+x:y\in B\}$. Now let $0\leq\theta<1$, $B=B(y,r)$ be any fixed ball, and $\textbf{e}$ be the unit vector in $\mathbb{R}^d$ in the direction of $y$. (If $y$ is the origin, we may take $\textbf{e}$ to be any unit vector in $\mathbb{R}^d$.) For $t\geq 0$, let 
$$B_t=B+\theta\sqrt{2\beta}t\textbf{e}.$$
Observe that $(B_t)_{t\geq 0}$ represents a ball of fixed size and time-dependent center that is moving away from the origin radially at a linear speed, but not moving faster than the BBM. Indeed, by the classical result of McKean \cite{MK1975}, it is well-known that the `speed' of strictly dyadic BBM in one dimension is equal to $\sqrt{2\beta}$, which was later generalized to higher dimensions by Engl\"ander and den Hollander \cite{E2003}. More precisely, we have:
\begin{thma}[Speed of BBM; \cite{MK1975,E2003}] 
Let $Z$ be a strictly dyadic BBM in $\mathbb{R}^d$. For $t\geq 0$ define $M_t:=inf\{r>0:\text{supp}(Z(t))\subseteq B(0,r)\}$ to be the radius of the minimal ball that contains the support of BBM at time $t$. Then, in any dimension, 
$$M_t/t \rightarrow \sqrt{2\beta} \quad \text{in probability} \quad \text{as} \quad t\rightarrow\infty.$$
\end{thma}
We remark that $M_t$ quantifies the spatial spread of BBM so that $M_t/t$ is a measure of the speed of BBM. More sophisticated results on the speed of BBM, such as almost sure results and higher order sublinear corrections, exist in the literature (see for example \cite{B1978,K2005}). For our purposes, Theorem A suffices; it says that typically for large $t$ and any $\varepsilon>0$, at time $t$ there will be particles inside $B(0,(\sqrt{2\beta}-\varepsilon)t)$ but no particles outside $B(0,(\sqrt{2\beta}+\varepsilon)t)$. Therefore, when we study the asymptotics of the mass in moving balls, to obtain meaningful results, we require that the ball moves in the `subcritical zone'(defined below). This explains why we consider $(B_t)_{t\geq 0}$, where the center of the ball $B_t$ is at a distance of $\theta\sqrt{2\beta}t+o(t)$ from the origin with the requirement that $0\leq\theta<1$.  
\begin{definition}[Subcritical zone]
For a family of Borel sets $(B_t)_{t\geq 0}$ in $\mathbb{R}^d$, we say that $(B_t)_{t\geq 0}$ moves in the \textit{subcritical zone} for $Z$ if there exists $0<\varepsilon<1$ and time $t_0$ such that 
$$B_t\subseteq B\left(0,\sqrt{2\beta}(1-\varepsilon)t\right) \quad \text{for all} \quad t\geq t_0.$$
\end{definition} 

The strong law of large numbers for local mass of BBM arises as a special case of \cite[Corollary, p.\ 222]{W1967}, where Watanabe established an almost sure result on the asymptotic behavior of certain branching Markov processes. The relevant special case of this result is as follows.
\begin{thmb}[SLLN for local mass of BBM; \cite{W1967}] 
Let $Z$ be a strictly dyadic BBM in $\mathbb{R}^d$. Then, for any fixed Borel set $B\subseteq \mathbb{R}^d$, 
$$\frac{Z_t(B)}{e^{\beta t}t^{-d/2}} \rightarrow (2\pi)^{-d/2}|B|\times W  \quad \text{as} \quad t\rightarrow\infty,$$
where $|B|$ is the Lebesgue measure of $B$, and $W$ is a $P$-a.s.\ strictly positive random variable.
\end{thmb} 
Theorem B can be viewed as a SLLN for local mass, because it says that with probability one, the mass in $B$ grows as its expectation as $t\rightarrow\infty$. Note that the expected mass in $B$ at time $t$ can be calculated using the first moment formula 
$$E[Z_t(B)]=E[Z_t(\mathbb{R}^d)]\times p_t(0,B)=e^{\beta t}\times \frac{1}{(2\pi t)^{d/2}}\int_{B}e^{-|x|^2/(2t)}\text{d}x,$$ 
where $p_t(0,B)$ denotes the probability that a Brownian particle starting at the origin falls inside $B$ at time $t$. Recall that $B_t=B+\theta\sqrt{2\beta}t\textbf{e}$, where $B=B(y,r)$ is a fixed ball, $\textbf{e}$ is the unit vector in the direction of $y$ (if $y$ is the origin, take $\textbf{e}$ to be any unit vector), and $0\leq\theta<1$. An extension of Theorem B for the local mass in moving Borel sets was given in \cite[Corollary 4]{B1992}, and implies that 
\begin{equation}\underset{t\rightarrow\infty}{\lim}\frac{1}{t}\log Z_t(B_t)=\beta(1-\theta^2) \quad \text{a.s.} \label{biggins}
\end{equation}
This means, the mass that falls inside $B_t$ is typically $\exp[\beta(1-\theta^2)t+o(t)]$. In this work, we mainly study the large deviations for local mass in a linearly moving ball $(B_t)_{t\geq 0}$ in the downward direction. More precisely, we are interested in the asymptotic behavior of 
$$P\left(Z_t(B_t)< e^{\beta at}\right) \quad \text{for} \quad 0\leq a<1-\theta^2,$$
where $a$ is an aytpically small exponent due to \eqref{biggins}. We then consider the mass that falls outside a linearly expanding ball $(\widehat{B}_t)_{t\geq 0}$ with $\widehat{B}_t:=B(0,x_t)$ and $x_t:=\theta\sqrt{2\beta}t$, and study the asymptotic behavior of 
$$P\left(Z_t(\widehat{B}_t^c)< e^{\beta at}\right) \quad \text{for} \quad 0\leq a<1-\theta^2,$$
where $\widehat{B}^c_t$ denotes the complement of $\widehat{B}_t$ in $\mathbb{R}^d$, and $a$ is again an aytpically small exponent for the mass in $\widehat{B}_t^c$ at time $t$.

\subsection{History} 

In the past fifty years, a variety of results have been obtained concerning the asymptotic behavior of the mass of BBM. It is well known that the total mass of BBM, that is, $N_t:=Z_t(\mathbb{R}^d)$ satisfies the following SLLN:
$$\underset{t\rightarrow\infty}{\lim}N_t e^{-\beta t}=M>0 \quad \text{a.s.},$$ 
meaning that the limit exists and is positive almost surely (see for example \cite[Thm.III.7.1]{AN1972}). In what follows, we refer to the mass that falls inside a time-dependent domain as \textit{local mass} if for each $t\geq 0$ the domain is bounded. Otherwise, we use the term \textit{non-local mass}. 

The first result on the SLLN for local mass of BBM goes back to \cite[Corollary, p. 222]{W1967}, where Watanabe established an almost sure result on the asymptotic behavior of certain branching Markov processes, from which the SLLN for local mass in fixed Borel sets emerged as a special case. This was improved by Biggins in \cite{B1992}, saying that almost surely the local mass in a linearly moving Borel set at time $t$ behaves as its expectation as $t\rightarrow\infty$. We note that the result of Biggins \cite[Corollary 4]{B1992} was originally cast in the discrete setting of a branching random walk in discrete time, and then extended in the same paper to the continuous setting of a BBM. Asymptotics of local mass of branching Markov processes other than BBM, involving more general motion components and branching mechanisms, have also been studied. In \cite{AH1976}, weak and strong laws of large numbers were proved for a certain class of branching Markov processes including branching diffusions. In \cite{EK2003}, an interesting dichotomy between local extinction and local exponential growth for certain branching diffusions (and superprocesses) was studied. More recently in \cite{CS2007} and \cite{EHK2010}, SLLN for a more general class of branching diffusions were proved. We refer the reader to \cite[Chapter 5]{E2007} for a brief survey on the LLN for local mass of spatial branching processes and superprocesses.

The growth of non-local mass of BBM, that is, the mass inside unbounded time-dependent domains in $\mathbb{R}^d$ has also been frequently studied; the most popular domain being the complement of a ball centered at the origin with a radius linearly growing in time but at a rate smaller than the speed of BBM. Recall that $M_t:=inf\{r>0:\text{supp}(Z(t))\subseteq B(0,r)\}$, and that by Theorem A for large $t$, typically there are no particles outside of $B(0,rt)$ for $r>\sqrt{2\beta}$, and there are particles outside of $B(0,\bar{r}t)$ for $\bar{r}<\sqrt{2\beta}$. In \cite{CR1988}, the asymptotics of the large deviation probabilities $P(M_t\geq rt)$ for $r>\sqrt{2\beta}$ was found in one dimension. Note that in this case $P(M_t\geq rt)$ is a probability of presence in a region where there would typically be no particles. In \cite{E2004}, the asymptotics of the large deviation probabilities $P(M_t\leq \bar{r}t)$ for $0<\bar{r}<\sqrt{2\beta}$ was found in any dimension. In this case, contrary to the event studied in \cite{CR1988}, since $0<\bar{r}<\sqrt{2\beta}$, $P(M_t\leq \bar{r}t)$ is a probability of absence in the region $(B(0,\bar{r}t))^c$ where there would typically be particles.  

Recently in \cite{S2017} and \cite{A2017}, asymptotic results have been obtained concerning the mass of BBM outside $B(0,rt)$, where $r$ is smaller than the typical $M_t/t$. In \cite{A2017}, upper tail asymptotics in one dimension for $Z_t([\theta\sqrt{2\beta}t,\infty))$ with $0<\theta<1$, were obtained for a strictly dyadic BBM with constant branching rate $\beta$. Due to \eqref{biggins}, the mass inside $[\theta\sqrt{2\beta}t,\infty)$ at time $t$ is typically $\exp[\beta(1-\theta^2)+o(t)]$; and in \cite{A2017}, large deviation probabilities $P(Z_t([\theta\sqrt{2\beta}t,\infty))\geq e^{\beta at})$ are studied for $1-\theta^2<a<1$. In \cite{S2017}, BBMs with space-dependent branching mechanisms and branching rate measures on $\mathbb{R}^d$ satisfying a certain Kato class condition were considered; first the speed (corresponding to typical $M_t/t$ for large $t$) of such processes was obtained, and then the SLLN for the mass outside $B(0,rt)$ was proved for $r$ smaller than the speed of the process.

\subsection{Motivation}

The motivation for present work comes from \cite[Lemma 1]{OE2018}, which was originally cast in the setting of a trap-avoiding problem, and which could be formulated in terms of the local mass of BBM as follows.
\begin{proposition} \label{prop}
Let $0<\theta<1$, and for $t\geq 0$ define $x_t=\theta\sqrt{2\beta}t$. Let $\rho:\mathbb{R}_+\to\mathbb{R}^d$ be a function, where $|\rho(t)|\in B(0,x_t)$ for all large $t$, and let $r>0$ be fixed. Abbreviate $B_t:=B(\rho(t),r)$ and define $\mathsf{S_t}:=\bigcap_{0\leq s\leq t}\left\{Z_s(B_s)=0\right\}$ to be the event that $Z$ has not hit $B=(B_s)_{s\geq 0}$ up to time $t$. Then, the non-hitting probability of $B$ by $Z$ up to time $t$ satisfies the following asymptotics:
\begin{equation}\underset{t\rightarrow\infty}{\limsup}\,\frac{1}{t}\log P\left(\mathsf{S_t}\right)\leq-\beta(1-\theta)\left(\sqrt{2}-1\right). \label{yee}
\end{equation}
\end{proposition} 
\begin{remark} In the context of trap-avoiding, the event $\mathsf{S_t}$ is the event of `survival up to time $t$' for the BBM from a single moving trap $B=(B_s)_{s\geq 0}$ of fixed size. 
\end{remark}
The present work was motivated by a search for a sharp upper bound that improves \eqref{yee}. This is achieved in the next section under Corollary~\ref{corollary2}, where a detailed comparison between Proposition~\ref{prop} and Corollary~\ref{corollary2} is also given.  
  
\smallskip
We conclude this section with an often used terminology and the outline of the paper.
\begin{definition}[SES]
A generic function $g: \mathbb R_+\to \mathbb R$ is called \emph{super-exponentially small (SES)} if $\lim_{t\to\infty}\log g(t)/t=-\infty$.
\end{definition}

{ \bf Outline:} The rest of the paper is organized as follows. In Section~\ref{section2}, we present our main results. In Section~\ref{section3}, we develop the preparation needed, including Lemma~\ref{lemma3} and its proof, for the proofs of Theorem~\ref{theorem1} and Theorem~\ref{theorem2}. In Section~\ref{section4}, we present the proof of Theorem~\ref{theorem1}, which is our main result. The optimization problem given in the statement of Theorem~\ref{theorem1} (see \eqref{ld}) is analyzed in Section~\ref{section5}, and the proof of Theorem~\ref{theorem2} is given in Section~\ref{section6}.

\section{Results}\label{section2}
We introduce further notation before stating our results. For a Borel set $A\subset\mathbb{R}^d$, we use $A^c=\mathbb{R}^d-A$ to denote its complement in $\mathbb{R}^d$. Also, for any event $E$ in some underlying set $\Omega$ of outcomes, we use $E^c$ to denote its complement event in $\Omega$. We use $c$ as a generic positive constant, whose value may change from line to line. If we wish to emphasize the dependence of $c$ on a parameter $p$, then we write $c_p$ or $c(p)$. We write $o(t)$ to refer to $g(t)$, where $g:\mathbb{R}_+\to \mathbb{R}$ is a generic function satisfying $\underset{t\rightarrow\infty}{\lim}g(t)/t=0$. 

Our main result is a large deviation result, giving the asymptotic exponential rate of decay for the probability that the mass of BBM inside a linearly moving ball of fixed size is atypically small on an exponential scale.

\begin{theorem}[Lower tail asymptotics for mass inside a moving ball] \label{theorem1}
Let $0\leq\theta<1$ and $B$ be a fixed ball in $\mathbb{R}^d$. Let $\textbf{e}$ be the unit vector in the direction of the center of $B$ if $B$ is not centered at the origin; otherwise let $\textbf{e}$ be any unit vector. For $t\geq 0$ define $B_t=B+\theta\sqrt{2\beta}t\textbf{e}$. Then, for $0\leq a<1-\theta^2$,
\begin{equation} \underset{t\rightarrow\infty}{\lim}\frac{1}{t}\log P\left(Z_t(B_t)<e^{\beta a t}\right)=-\beta \times I(\theta,a), \label{ld0}
\end{equation}
where
\begin{equation}
I(\theta,a)=\underset{\rho\in(0,\bar\rho]}{\inf}\left[\rho+\frac{\left(\sqrt{(1-\rho)^2-a(1-\rho)}-\theta\right)^2}{\rho}\right], \label{ld}
\end{equation}
and
$$\bar\rho=\bar\rho(\theta,a)=1-a/2-\sqrt{(a/2)^2+\theta^2}.$$
\end{theorem}

\begin{remark} In terms of the BBM's optimal strategies for realizing the large deviation event 
\newline $\left\{Z_t(B_t)< e^{\beta a t}\right\}$, this means (see the proof of Theorem~\ref{theorem1} for details) to realize $\left\{Z_t(B)< e^{\beta a t}\right\}$: \newline
the system suppresses the branching completely, and sends the single particle to a distance of $(\sqrt{(1-\hat{\rho})^2-a(1-\hat{\rho})}-\theta)t+o(t)$ in the opposite direction of the center of $B_t$ over $[0,\hat{\rho}t]$, and then behaves `normally' in the remaining interval $[\hat{\rho}t,t]$,  
where $\hat\rho$ denotes the unique minimizer (see Proposition~\ref{prop5}) of the optimization problem in \eqref{ld}. 	
\end{remark}

Two corollaries of interest follow immediately from Theorem~\ref{theorem1} as special cases. Setting $\theta=0$ in Theorem~\ref{theorem1} covers the case of fixed balls, and yields Corollary~\ref{corollary1}. Next, setting $a=0$ in Theorem~\ref{theorem1} gives the asymptotic exponential rate of decay for the probability that no particle of BBM falls inside a linearly moving ball of fixed size, and yields Corollary~\ref{corollary2}. The optimization problem on the right-hand side of \eqref{ld} cannot be solved to find a closed-form expression for $I(\theta,a)$ in the general case, where both $\theta$ and $a$ are nonzero; we analyze this case in Section~\ref{section5}.  

\begin{corollary}[Lower tail asymptotics for mass inside a fixed ball] \label{corollary1}
Let $B\subseteq \mathbb{R}^d$ be any fixed ball. Then, for $0\leq a<1$,
\begin{equation} \underset{t\rightarrow\infty}{\lim}\frac{1}{t}\log P\left(Z_t(B)<e^{\beta a t}\right)= \begin{cases}
-\beta\left[2\sqrt{2(1-a)}-2+a \right], &  \:0\leq a<1/2, \\
-\beta(1-a), & \:1/2\leq a<1.
\end{cases}  \label{cor1}
\end{equation}
\end{corollary} 

\begin{proof} Set $\theta=0$ in \eqref{ld}. Then, $\bar\rho=1-a$, and over $\rho\in(0,1-a]$ the functional on the right-hand side of \eqref{ld} has a unique minimizer at $\hat{\rho}=\sqrt{(1-a)/2}$ for $0\leq a <1/2$, and at $\hat{\rho}=1-a$ for $1/2\leq a<1$.  
\end{proof}

\begin{remark} Corollary~\ref{corollary1} says that there is a continuous phase transition at $a=1/2$ in the asymptotic behavior of $P\left(Z_t(B)< e^{\beta a t}\right)$, which is revealed by the Lyapunov exponent in \eqref{cor1} (see Figure 1). In terms of the BBM's optimal strategies for realizing the large deviation event $\left\{Z_t(B)< e^{\beta a t}\right\}$, this means to realize $\left\{Z_t(B)< e^{\beta a t}\right\}$:
\begin{itemize}
	\item for $0\leq a<1/2$, the system suppresses the branching completely, and sends the single particle to a distance of $((1-\hat{\rho})^2-a(1-\hat{\rho}))t+o(t)$ away from the origin over $[0,\hat{\rho}t]$, and then behaves `normally' in the remaining interval $[\hat{\rho}t,t]$;  
	\item for $1/2\leq a<1$, the system only suppresses the branching completely over $[0,\hat{\rho}t]$, and behaves `normally' otherwise. This means, the parameter $a$ is high enough so that there is no need to initially move the single particle to a linear distance away from $B$ in order to realize $\left\{Z_t(B)< e^{\beta a t}\right\}$.  
\end{itemize}	
\end{remark}

\vspace{1cm}

\setlength{\unitlength}{0.4cm}
\begin{picture}(36,10)
	\put(0,0){\line(12,0){12}}
  \put(0,0){\line(0,10){10}}
  \put(-1,-.3){$0$}
  \put(12.5,-.3){$a$}
	\put(-.3,11){$\hat\rho(a)$}
	\put(-1.5,10){(i)}
  {\thicklines
   \qbezier(0,8)(2.5,7.7)(5,5)
   \qbezier(5,5)(7,3)(9.9,0.1)
  }
  \qbezier[30](0,5)(3,5)(5,5)
  \qbezier[30](5,0)(5,3)(5,5)
  \put(5,5){\circle*{.25}}
	\put(10,0){\circle{.25}}
  \put(4.5,-1){$1/2$}
  \put(-2,4.5){$1/2$}
  \put(9.5,-1){$1$}
	\put(-2.7,7.5){$\sqrt{2}/2$}
	
	\put(18,0){\line(12,0){12}}
  \put(18,0){\line(0,10){10}}
  \put(17,-.3){$0$}
  \put(30.5,-.3){$a$}
	\put(17.7,11){$I(0,a)$}
	\put(16,10){(ii)}
  {\thicklines
   \qbezier(18,8)(20.5,7.5)(23,5)
   \qbezier(23,5)(25,3)(27.9,0.1)
  }
  \qbezier[30](18,5)(21,5)(23,5)
  \qbezier[30](23,0)(23,3)(23,5)
  \put(23,5){\circle*{.25}}
	\put(28,0){\circle{.25}}
  \put(22.5,-1){$1/2$}
  \put(16,4.5){$1/2$}
  \put(27.5,-1){$1$}
	\put(13.8,7.5){$2(\sqrt{2}-1)$}
  
	\put(-1.8,-2.5)
  {\small
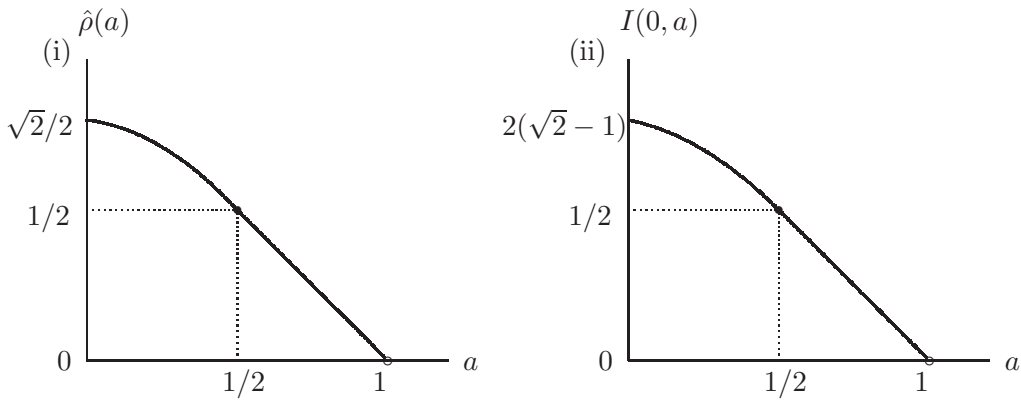
   Figure\ 1.~ Qualitative plot of: (i) $a\mapsto \hat\rho(a)$, (ii) $a\mapsto I(0,a)$ when $\theta=0$. Phase transition occurs at $a=1/2$.
  \normalsize
  }
\end{picture}
 
\vspace{1.5cm}

\begin{corollary}[Asymptotic probability of no particle inside a moving ball] \label{corollary2}
Let $0\leq\theta<1$ and $B$ be a fixed ball in $\mathbb{R}^d$. Let $\textbf{e}$ be the unit vector in the direction of the center of $B$ if $B$ is not centered at the origin; otherwise let $\textbf{e}$ be any unit vector. For $t\geq 0$ define $B_t=B+\theta\sqrt{2\beta}t\textbf{e}$. Then, 
\begin{equation} \underset{t\rightarrow\infty}{\lim}\:\frac{1}{t}\log P\left(Z_t(B_t)=0\right)=-2\beta(\sqrt{2}-1)(1-\theta).
\end{equation}
\end{corollary}

\begin{proof} Set $a=0$ in \eqref{ld}. Then, $\bar\rho=1-\theta$, and over $\rho\in(0,1-\theta]$ the functional on the right-hand side of \eqref{ld} has a unique minimizer at $\hat\rho=(1-\theta)/\sqrt{2}$.
\end{proof}

As stated in the introduction, the starting point of the present work is \cite[Lemma 1]{OE2018}, which is equivalently formulated in terms of local mass of BBM in Proposition~\ref{prop}. Corollary~\ref{corollary2} improves Proposition~\ref{prop} in two respects:
\begin{itemize}
	\item it sharpens the Lyapunov exponent in Proposition~\ref{prop} from $\beta(\sqrt{2}-1)\theta$ to $2\beta(\sqrt{2}-1)\theta$,
	\item in Corollary~\ref{corollary2}, the large deviation event is $\left\{Z_t(B_t)=0\right\}$, whereas in Proposition~\ref{prop} it was $\bigcap_{0\leq s\leq t}\left\{Z_s(B_s)=0\right\}$. Clearly, $\bigcap_{0\leq s\leq t}\left\{Z_s(B_s)=0\right\}\subseteq \left\{Z_t(B_t)=0\right\}$. 
\end{itemize}

Next, we have a large deviation result, giving the asymptotic exponential rate of decay for the probability that the mass of BBM outside a ball with a linearly growing radius and fixed center, is atypically small on an exponential scale.

\begin{theorem}[Lower tail asymptotics for mass outside a linearly expanding ball] \label{theorem2}
Let $0<\theta<1$, and for $t\geq 0$ let $x_t=\theta\sqrt{2\beta}t$ and denote $\widehat{B}_t=B(0,x_t)$. Then, for $0\leq a<1-\theta^2$,
\begin{equation} \underset{t\rightarrow\infty}{\lim}\frac{1}{t}\log P\left(Z_t(\widehat{B}_t^c)< e^{\beta a t}\right)=-\beta\left[1-a/2-\sqrt{\theta^2+(a/2)^2}\right]. \label{ld2}
\end{equation}
\end{theorem} 

Note that by setting $a=0$ in \eqref{ld2}, we obtain the asymptotic rate of decay for the probability of confinement of BBM into `subcritical balls' \textit{at time $t$}, which expectedly turns out to be identical to that of probability of confinement of BBM into `subcritical balls' \textit{throughout the interval $[0,t]$} (see \cite[Thm.2]{E2004}).

\section{Preparations}\label{section3}

In this section, we first list a few well-known results concerning the global growth of branching systems, and asymptotic probability of atypically large Brownian displacements. These results will be useful in the proofs of the main theorems and Lemma~\ref{lemma3}. Then, we state and prove Lemma~\ref{lemma3}, which will be the crucial ingredient in the proofs of the upper bounds of Theorem~\ref{theorem1} and Theorem~\ref{theorem2}, which are done via bootstrap arguments that are composed of two steps. Lemma~\ref{lemma3} constitutes the first of these steps, and can be viewed as a weaker form of Theorem~\ref{theorem1}. 

The following two propositions are standard in the theory of branching processes. For the proofs of Proposition~\ref{proposition1} and Proposition~\ref{proposition2}, see for example \cite[Sect.8.11]{KT1975} and \cite[Sect.8.8]{KT1975}, respectively.

\begin{proposition}[Distribution of mass in branching systems]\label{proposition1}
For a strictly dyadic continuous-time branching process $N=(N(t))_{t\geq 0}$ with constant branching rate $\beta>0$, the probability distribution at time $t$ is given by 
\begin{equation} P(N(t)=k)=e^{-\beta t}(1-e^{-\beta t})^{k-1},\quad k\geq 1, \nonumber
\end{equation}
from which it follows that
\begin{equation} P(N(t)>k)=(1-e^{-\beta t})^k  \label{prop1}.
\end{equation}
\end{proposition}

\begin{proposition}[Expected mass in branching systems]\label{proposition2}
For a strictly dyadic continuous-time branching process $N=(N(t))_{t\geq 0}$ with constant branching rate $\beta>0$, the expected mass at time $t$ is given by 
\begin{equation} E[N(t)]=e^{\beta t}. \label{prop2}
\end{equation}
\end{proposition}

The following result is taken from \cite{OCE2017} (see Lemma 5 and its proof therein). 

\begin{proposition}[Linear Brownian displacements]\label{proposition3}
Let $X=(X(t))_{t\geq 0}$ represent a standard $d$-dimensional Brownian motion starting at the origin, and $\mathbf{P}_0$ the corresponding probability. Then, for $\gamma>0$,
\begin{equation}  \mathbf{P}_0\left(\underset{0\leq s\leq t}{\sup}|X(s)|>\gamma t\right)=\exp[-\gamma^2t/2+o(t)].  \label{prop3}
\end{equation}
\end{proposition}

The term `overwhelming probability' is used with a precise meaning throughout the paper, which is given as follows.
\begin{definition}[Overwhelming probability]
Let $(A_t)_{t>0}$ be a family of events indexed by time $t$. We say that $A_t$ occurs \emph{with overwhelming probability} as $t\rightarrow\infty$ if there is a constant $c>0$ and time $t_0$ such that 
$$P(A_t^c)\leq e^{-ct} \quad \text{for all} \quad t\geq t_0. $$ 
\end{definition}

We now state and prove Lemma~\ref{lemma3}.

\begin{lemma}\label{lemma3} 
Let $0\leq\theta<1$, $B$ be a fixed ball in $\mathbb{R}^d$, and $\textbf{e}$ be the unit vector in the direction of the center of $B$ if $B$ is not centered at the origin; otherwise let $\textbf{e}$ be any unit vector. For $t\geq 0$ define $B_t=B+\theta\sqrt{2\beta}t\textbf{e}$. Then, for each $0\leq a<1-\theta^2$, there exists a constant $c=c(\beta,\theta,a)>0$ such that 
\begin{equation} \underset{t\rightarrow\infty}{\limsup}\:\frac{1}{t}\log P\left(Z_t(B_t)< e^{\beta at}\right)\leq-c. \nonumber
\end{equation}
\end{lemma}

\begin{proof}
The strategy is to split the time interval $[0,t]$ into two pieces so that with overwhelming probability, sufficiently many particles are created and kept close enough to the origin in the first piece, and then in the second piece at least one of these particles initiates a sub-BBM that contributes at least $\exp(\beta at)$ particles to $B_t$ at time $t$.

To start the proof, split the time interval $[0,t]$ into two pieces as $[0,\delta t]$ and $[\delta t,t]$, where $0<\delta<1$ is a number which will later depend on $\theta$ and $a$. 

\underline{Step 1}: For $0\leq a<1-\theta^2$ and $t>0$, let  
$$ A_t:=\left\{Z_t(B_t)< e^{\beta at}\right\}.$$
Next, for $y>0$ and $t>0$, let $y(t)=y\sqrt{2\beta}t$, and define the events
$$ E_t:=\left\{N_{\delta t}\geq t^2\right\}, \quad F_t:=\left\{Z_{\delta t}\left(B(0,y(t))\right)\geq t^2\right\}.$$
Estimate
\begin{equation} 
P(A_t)\leq  P(A_t \mid F_t)+P(F_t^c \mid E_t)+P(E_t^c). 
\label{eq113}
\end{equation}
Using \eqref{prop1}, we bound $P(E_t^c)$ from above as 
\begin{equation} P(E_t^c)\leq \exp[-\beta \delta t+o(t)]. \label{eq114}
\end{equation}
It is clear that for an upper bound on $P(F_t^c \mid E_t)$, we may assume that the BBM has exactly $\left\lceil t^2\right\rceil$ particles at time $\delta t$. Therefore, an upper bound for $P(F_t^c \mid E_t)$ is given by the probability that at least one among the $\left\lceil t^2\right\rceil$ particles present at time $\delta t$ has escaped $B(0,y(t))$ during $[0,\delta t]$. Since motion and branching mechanisms are independent; even conditioned on $E_t$, the position of each particle present at time $\delta t$ is identically distributed as $X(\delta t)$, where $X=(X(s))_{s\geq 0}$ represents the standard $d$-dimensional Brownian motion. It follows from the union bound and Proposition~\ref{proposition3} that 
\begin{equation} P(F_t^c \mid E_t)\leq \left\lceil t^2\right\rceil\exp\left[-\beta\frac{y^2}{\delta}t+o(t)\right]=\exp\left[-\beta\frac{y^2}{\delta}t+o(t)\right]. \label{eq115}
\end{equation}
On the event $F_t$, there are at least $\left\lceil t^2\right\rceil$ particles inside $B(0,y(t))$ at time $\delta t$, and by triangle inequality each of them is at most at a distance of
$$(\theta+y)\sqrt{2\beta}t+o(t)=:r(t) $$
from $B_t$. Here, the distance between a Borel set $B$ and a point $y\in\mathbb{R}^d$ is defined as $\underset{x\in B}{\inf}|y-x|$.

\underline{Step 2}: Apply the branching Markov property at time $\delta t$. By \eqref{biggins}, for any $\varepsilon>0$, with probability at least $1-\varepsilon$, the sub-BBM initiated by a particle in $B(0,y(t))$ at time $\delta t$, independently of all others, contributes at least 
$$\exp\left[(1-\varepsilon)\beta t\left((1-\delta)-\frac{(\theta+y)^2}{1-\delta}\right)\right]$$ particles to $B_t$ at time $t$. To be precise, let $S(\delta t)=\left\{X_1,X_2,\ldots,X_{N_{\delta t}}\right\}$ denote the positions of particles present at time $\delta t$, where we suppress the dependence of $X_j$ on $t$ for brevity. Abbreviate 
$$m_t:=\left\lceil t^2\right\rceil.$$ 
Conditioned on the event $F_t$, we may and do order the elements in $S(\delta t)$ so that $|X_j|<y(t)$ for all $1\leq j \leq m_t$. Let $Z^j$ denote the BBM starting at $X_j$, running over a time $(1-\delta)t$. Then, \eqref{biggins} implies that for any $\varepsilon>0$, for all large $t$,
\begin{equation} P_{X_j}\left(Z^j_{(1-\delta)t}(B_t)< \exp\left[(1-\varepsilon)\beta t\left((1-\delta)-\frac{(\theta+y)^2}{1-\delta}\right)\right]\right)\leq \varepsilon \nonumber
\end{equation}
uniformly over $j=1,\ldots,m_t$. Choose $\varepsilon$, $\delta$ and $y$ small enough to satisfy
\begin{equation} (1-\varepsilon)\left((1-\delta)-\frac{(\theta+y)^2}{1-\delta}\right)>a. \label{cond2}
\end{equation}
Then, for $A_t$ to occur conditional on $F_t$, it is necessary that each of the at least $m_t$ many particles present in $B(0,y(t))$ at time $\delta t$ should contribute less than 
$\exp\left[(1-\varepsilon)\beta t\left((1-\delta)-\frac{(\theta+y)^2}{1-\delta}\right)\right]$  
particles to $B_t$ at time $t$. By independence of particles, this has probability bounded above by 
$$ \varepsilon^{t^2}=\exp[-\log(1/e)t^2], $$
which is SES in $t$. Along with \eqref{eq113}-\eqref{eq115}, this implies that there exists $c=c(\beta,\theta,a)>0$ and $t_0$ such that for all $t\geq t_0$,
$$P(A_t)\leq e^{-ct}.$$  
We may choose $\varepsilon$, $\delta$ and $y$ all small enough so as to satisfy the condition in $\eqref{cond2}$ since $a<1-\theta^2$; take for instance $\varepsilon=\delta=y=(1-\theta^2-a)/2$. This completes the proof. 
\end{proof}

The proof of the upper bound of \eqref{ld0} in Theorem~\ref{theorem1} uses a bootstrap argument given in two steps. The first step is completed by Lemma~\ref{lemma3} above, and the second step will be given below in the next section. In this respect, Lemma~\ref{lemma3} completes the preparation needed for the proof of Theorem~\ref{theorem1}.

\section{Proof of Theorem~\ref{theorem1}} \label{section4}
This section is devoted to the proof of Theorem~\ref{theorem1}. For the lower bound, we simply find a strategy that realizes the desired event with optimal probability on an exponential scale. The proof of the upper bound is more technical, where we effectively show that the strategy that gave the lower bound is indeed optimal, that is, there is no better strategy that realizes the desired event at a smaller exponential cost of probability. The proof of the upper bound that is given below uses a method similar to that of \cite[Thm.1]{OCE2017}, and can be viewed as the second step of the bootstrap argument, whose first step was completed by Lemma~\ref{lemma3}.  
  
\subsection{Proof of the lower bound}
The lower bound arises from the following strategy which gives a mass of less than $\exp(\beta at)$ to $B_t$ at time $t$. Define
\begin{equation}
\bar\rho=\bar\rho(\theta,a)=1-a/2-\sqrt{(a/2)^2+\theta^2}, \label{barrho}
\end{equation}
which is chosen so that $(1-\bar\rho)^2-a(1-\bar\rho)=\theta^2$.
Let $0<\rho\leq \bar\rho$ and $\varepsilon>0$. First, in the time interval $[0,\rho t]$, suppress branching completely and move the single Brownian particle to a distance of 
$$d(t):=\left(\sqrt{(1-\rho)^2-a(1-\rho)}-\theta+\varepsilon\right)\sqrt{2\beta}t+o(t)$$
from the origin in the opposite direction of $\textbf{e}$, where $\textbf{e}$ is the direction of motion of $(B_t)_{t\geq 0}$. By Proposition~\ref{proposition3}, this partial strategy has probability
\begin{align}& \:\:\:\exp(-\beta\rho t)\exp\left[-\frac{((\sqrt{(1-\rho)^2-a(1-\rho)}-\theta+\varepsilon)\sqrt{2\beta})^2}{2\rho}t+o(t)\right] \nonumber \\
&= \exp\left[-\beta\left(\rho+\frac{(\sqrt{(1-\rho)^2-a(1-\rho)}-\theta+\varepsilon)^2}{\rho}\right) t+o(t)\right] \label{strategy}
\end{align}
where the first exponential factor comes from suppressing the branching, and the second from the linear Brownian displacement. Next, in the remaining interval $[\rho t,t]$, prevent the BBM from sending at least $\exp(\beta at)$ particles to $B_t$ at time $t$. This comes for `free' and has probability $\exp[o(t)]$ since it is realized when the BBM behaves `normally.' Indeed, the distance between $B_t$ and the position of the single particle at time $\delta t$ is
$$d(t)+\theta \sqrt{2\beta}t+o(t)=\left(\sqrt{(1-\rho)^2-a(1-\rho)}+\varepsilon\right)\sqrt{2\beta}t+o(t).$$
This distance is too large for the sub-BBM emanating from the particle at time $\delta t$ to contribute a mass of $\exp(\beta at)$ to $B_t$ at time $t$ due to \eqref{biggins} since 
$$(1-\rho)-\frac{\left(\sqrt{(1-\rho)^2-a(1-\rho)}+\varepsilon\right)^2}{1-\rho}<a.$$
(For $a=0=\theta$, we have $\bar\rho=1$. Therefore, take $\rho<\bar\rho$.) Then, applying the Markov property at time $\delta t$, using \eqref{strategy}, and letting $\varepsilon\rightarrow 0$ yields
\begin{equation} \underset{t\rightarrow\infty}{\liminf}\:\frac{1}{t}\log P\left(Z_t(B_t)< e^{\beta at}\right)\geq -\beta\left[\rho+\frac{(\sqrt{(1-\rho)^2-a(1-\rho)}-\theta)^2}{\rho}\right]. \label{eq20}
\end{equation}  
Finally, optimize the right-hand side of \eqref{eq20} over $\rho\in(0,\bar\rho]$; equivalently, minimize 
$$f_{\theta,a}(\rho):=\rho+\frac{(\sqrt{(1-\rho)^2-a(1-\rho)}-\theta)^2}{\rho}$$
over $\rho\in(0,\bar\rho]$, which completes the proof of the lower bound.
(If $\bar\rho<\rho\leq 1$, then there is no need to move the initial particle to a linear distance from the origin; suppressing the branching alone over $[0,\rho t]$ and letting the BBM behave normally over $[\rho t,t]$ realizes the event $\{Z_t(B_t)< e^{\beta at}\}$. However, this strategy cannot be optimal, because it has probability $\exp(-\beta \rho t)$ with $\rho>\bar\rho$, which is smaller than $\exp[-\beta(\min_{\rho\in(0,\bar\rho]}f_{\theta,a}(\rho)) t]$ on a logarithmic scale.)

\smallskip

\subsection{Proof of the upper bound}

A bootstrap argument is given for the proof of the upper bound, where the central ingredient is Lemma~\ref{lemma3}. Namely, we have already shown in Lemma~\ref{lemma3} that for large $t$ the probability of exponentially few particles in $B_t$ at time $t$ decays at least as fast as $\exp(-ct)$ for some $c=c(\beta,\theta,a)>0$. Here, using this, we show that the decay constant $c(\beta,\theta,a)$ can actually be improved to $\beta \times I(\theta,a)$, where $I$ is given in \eqref{ld}.  

We split the time interval $[0,t]$ into two pieces at $\rho t$, $\rho \in [0,1]$, which is the instant at which the total mass exceeds $\lfloor t \rfloor$. In the first piece, the branching is partially suppressed to give polynomially many particles only, which has an exponential probabilistic cost; whereas we are able to keep all of these particles close enough to the origin (at sublinear distance) at no cost since there are not exponentially many of them. In the second piece, given that we now have $\lfloor t \rfloor$ particles close enough to the origin, we argue that with overwhelming probability, at least one of these particles initiates a sub-BBM, that evolves in $[\rho t,t]$ and eventually sends at least $\exp(\beta at)$ particles to $B_t$ at time $t$. To catch the optimal $\rho$, we discretize $[0,t]$ into many small pieces, and condition the process on in which piece $\rho t$ falls, which results in a sum of terms, of which only the largest contributes on an exponential scale. 

Recall that $N_t=Z_t(\mathbb{R}^d)$, and for $t>1$ define the random variable
\begin{equation} \rho_t=\sup\left\{\rho \in [0,1]:N_{\rho t}\leq \lfloor t\rfloor \right\}. \nonumber
\end{equation}
Observe that for $x\in[0,1]$, we have $\{\rho_t\geq x \}\subseteq \{N_{xt}\leq \lfloor t\rfloor+1 \}$. We start by conditioning on $\rho_t$. For $t>0$ define the event
$$A_t:=\{ Z_t(B_t)< e^{\beta at} \},$$
and recall the definition of $\bar\rho$ from \eqref{barrho}. Then, for every $n=1,2,3,\ldots$
\begin{align} P(A_t)&\:=\sum_{i=0}^{\lfloor \bar\rho n-1\rfloor-1} P\left(A_t\cap \left\{\frac{i}{n}\leq \rho_t <\frac{i+1}{n}\right\} \right)+ P\left(A_t\cap \left\{\rho_t\geq \frac{\lfloor \bar\rho n-1\rfloor}{n}\right\} \right) \nonumber \\
&\:\:\leq \sum_{i=0}^{\lfloor \bar\rho n-1\rfloor-1} \exp\left[-\beta \frac{i}{n}t+o(t)\right]P^{(i,n)}_t\left(A_t\right)+\exp\left[-\beta \left(\bar\rho-\frac{2}{n}\right) t+o(t)\right], \label{eq21}
\end{align}
where we use \eqref{prop1}, which yields $P(N_{(i/n)t}\leq \lfloor t\rfloor+1)=\exp[-\beta(i/n)t+o(t)]$, to control \newline
$P(\frac{i}{n}\leq \rho_t <\frac{i+1}{n})$, and introduce the conditional probabilities
\begin{equation} P^{(i,n)}_t(\cdot)=P\left(\ \cdot \ \middle| \ \frac{i}{n}\leq \rho_t <\frac{i+1}{n}\right), \quad i=0,1,\ldots,\lfloor \bar\rho n-1\rfloor-1. \nonumber
\end{equation}

For each pair $(i,n)$ with $n=1,2,3,\ldots$ and $i=0,1,\ldots,\lfloor \bar\rho n-1\rfloor-1$, define the interval 
$$I^{(i,n)}:=[i/n,(i+1)/n),$$
and the radius
\begin{equation}
r_t^{(i,n)}:=\sqrt{2\beta}\left(\sqrt{\left(1-\frac{i+1}{n}\right)^2-a\left(1-\frac{i+1}{n}\right)}-\theta-\varepsilon\right)t.\label{eqradius}
\end{equation}
By definition of $\rho_t$, conditional on the event $\rho_t\in I^{(i,n)}$, there exists an instant in $[ti/n,t(i+1)/n)$, namely $\rho_t t$, at which there are exactly $\lfloor t\rfloor+1$ particles in the system. Let $\varepsilon=\varepsilon(n)>0$ be small enough so that \eqref{eqradius} is positive for each $i=0,1,\ldots,\lfloor \bar\rho n-1\rfloor-1$, and let $E_t^{(i,n)}$ be the event that among the $\lfloor t\rfloor+1$ particles alive at $\rho_t t$, there is at least one outside $B_t^{(i,n)}:=B\left(0,r_t^{(i,n)}\right)$. Estimate
\begin{equation} P^{(i,n)}_t\left(A_t\right)\leq P^{(i,n)}_t\left(E_t^{(i,n)}\right)+P^{(i,n)}_t\left(A_t \mid [E_t^{(i,n)}]^c\right). \label{eq22}
\end{equation}
On the event $[E_t^{(i,n)}]^c$, there are $\lfloor t\rfloor+1$ particles in $B_t^{(i,n)}$ at time $\rho_t t$. Then, conditional on $\rho_t\in I^{(i,n)}$, Lemma~\ref{lemma3} and \eqref{eqradius} imply that with overwhelming probability, the sub-BBM emanating from each such particle at time $\rho_t t$ evolves in the remaining time of length at least $(1-(i+1)/n)t$ to contribute at least $\exp(\beta at)$ particles to $B_t$ at time $t$. This is because
\begin{equation}
a<\left(1-\frac{i+1}{n}\right)-\frac{\left(\sqrt{\left(1-\frac{i+1}{n}\right)^2-a\left(1-\frac{i+1}{n}\right)}-\varepsilon\right)^2}{1-\frac{i+1}{n}}, \label{eq25}
\end{equation}
that is, the distance between $B_t$ and the starting point of the sub-BBM is too short for the sub-BBM to send less than $\exp(\beta at)$ particles to $B_t$ in the remaining time, which is at least $(1-(i+1)/n)t$.  
More precisely, let $p_t^y$ be the probability that a BBM starting with a single particle at position $y\in\mathbb{R}^d$ contributes less than $\exp(\beta at)$ particles to $B_t$ at time $t$. Then since $\rho_t t<t(i+1)/n$ conditioned on $\rho_t\in I^{(i,n)}$, by Lemma~\ref{lemma3} and \eqref{eq25}, uniformly over $y\in B_t^{(i,n)}$ there exists a constant $c>0$ and $t_0$ such that
$$p_{t(1-\rho_t)}^y \leq e^{-ct} \quad \text{for all} \quad t\geq t_0. $$   
Hence, by the strong Markov property applied at time $\rho_t t$ and the independence of particles present at that time, for all large $t$ we have
\begin{equation} P^{(i,n)}_t\left(A_t \mid [E_t^{(i,n)}]^c\right)\leq \left(e^{-ct}\right)^{\lfloor t\rfloor+1}\leq e^{-ct^2}, \label{kedibaba}
\end{equation}
which is SES in $t$.
Now consider the first term on the right-hand side of (\ref{eq22}). Recall that if the event $E_t^{(i,n)}$ occurs, then at least one among $\lfloor t \rfloor+1$ many particles has escaped $B_t^{(i,n)}$ by time at most $t(i+1)/n$. Note that each particle alive at time $s$ is at a random point, whose spatial distribution is identical to that of $X(s)$, where $X=(X(t))_{t\geq 0}$ denotes the standard Brownian motion in $d$ dimensions. Therefore, by Proposition~\ref{proposition3} and the union bound, we have
\begin{equation} P^{(i,n)}_t\left(E_t^{(i,n)}\right)\leq \: (\lfloor t \rfloor+1)\,\exp\left[-\frac{\left(r_t^{(i,n)}\right)^2}{2(i+1)/n}t+o(t)\right]
= \: \exp\left[-\frac{\left(r_t^{(i,n)}\right)^2}{2(i+1)/n}t+o(t)\right] .  \label{eq23}
\end{equation}  
From \eqref{eqradius}, \eqref{eq22}, \eqref{kedibaba} and \eqref{eq23}, we obtain
\begin{equation} P^{(i,n)}_t\left(A_t\right) \leq \exp\left[-\beta\,\frac{\left(\sqrt{\left(1-\frac{i+1}{n}\right)^2-a\left(1-\frac{i+1}{n}\right)}-\theta-\varepsilon\right)^2}{(i+1)/n}\,t+o(t)\right]+e^{-ct^2}. \label{eq24}
\end{equation}
Substituting \eqref{eq24} into \eqref{eq21}, and optimizing over $i\in\left\{0,1,\ldots,\lfloor \bar\rho n-1\rfloor-1\right\}$ gives
\begin{align} &\underset{t\rightarrow \infty}{\limsup}\:\frac{1}{t}\log P\left(A_t\right) 
\leq \nonumber \\
& -\beta\left[\underset{i\in\left\{0,1,\ldots,\lfloor \bar\rho n-1\rfloor-1\right\}}{\min}
\left\{\frac{i}{n}+\frac{\left(\sqrt{\left(1-\frac{i+1}{n}\right)^2-a\left(1-\frac{i+1}{n}\right)}-\theta-\varepsilon\right)^2}{(i+1)/n}\right\}\wedge\left(\bar\rho-\frac{2}{n}\right)\right], \label{big}
\end{align}
where we use $a\wedge b$ to denote the minimum of $a$ and $b$. Now first let $\varepsilon \rightarrow 0$, then set $\rho=i/n$, let $n\rightarrow \infty$, and use the continuity of the functional form from which the minimum is taken to obtain
\begin{equation} \underset{t\rightarrow\infty}{\limsup}\:\frac{1}{t}\log P\left(A_t\right)\leq -\beta\underset{\rho\in(0,\bar\rho]}{\inf}\left[\rho+\frac{\sqrt{(1-\rho)^2-a(1-\rho)}-\theta}{\rho}\right].  \label{big2}
\end{equation}
(Note that we have not written the last term on the right-hand side of \eqref{big} explicitly in \eqref{big2}, because once $n\rightarrow \infty$, this term becomes $\bar\rho$, which is attained by the function inside the infimum on the right-hand side of \eqref{big2} if we set $\rho=\bar\rho$.)

\section{Analysis of the optimization problem} \label{section5}
Here, we analyze the optimization problem given in the statement of Theorem~\ref{theorem1}:
\begin{equation}
I(\theta,a)=\underset{\rho\in(0,\bar\rho]}{\inf}\left[\rho+\frac{\left(\sqrt{(1-\rho)^2-a(1-\rho)}-\theta\right)^2}{\rho}\right]. \label{I}
\end{equation}
Let $\theta$ and $a$ be such that $0\leq a<1-\theta^2\leq 1$, and for $\rho\in(0,1-a]$ define the function $f_{\theta,a}$ by 
\begin{equation}
f_{\theta,a}(\rho)=\rho+\left(\sqrt{(1-\rho)^2-a(1-\rho)}-\theta\right)^2/\rho \label{f} 
\end{equation}
so that
$$I(\theta,a)=\underset{\rho\in(0,\bar\rho]}{\inf}f_{\theta,a}(\rho).$$
Here we consider $f$ on the larger interval $(0,1-a]$ as opposed to $(0,\bar\rho]$. This is harmless, because for each pair $(\theta,a)$, the minimizer of \eqref{f} will be shown to exist, be unique over $\rho\in(0,1-a]$, and be at most $\bar\rho$. For each pair $(\theta,a)$, denote the minimizer by $\hat\rho=\hat\rho(\theta,a)$. (For notational convenience, when $\theta$ or $a$ is fixed in the analysis below, we suppress the dependence of $f$ and $\hat\rho$ on that parameter in the notation. We also occasionally write simply $f$ and $\hat\rho$, and suppress their dependence on both $\theta$ and $a$.) The optimization problem in \eqref{I} can be solved to find closed-form formulas for $\hat\rho$ and $f(\hat\rho)$ when $\theta=0$ and when $a=0$; these were found in Corollary~\ref{corollary1} and Corollary~\ref{corollary2}, respectively. Therefore, here we suppose that 
$$0<a<1-\theta^2<1.$$

Observe that for each pair $(\theta,a)$, $f$ is differentiable on $(0,1-a)$, and its derivative, which we denote by $f'$, has the following rule:
\begin{align} f_{\theta,a}'(\rho)=&\: \frac{1}{\rho^2}\left[(1-\theta^2-a)-2(1-a-\rho^2)+\theta\,\frac{(1-\rho)(1-a-\rho)+(1-a-\rho^2)}{\sqrt{(1-\rho)(1-a-\rho)}}\right]  \label{derivative} \\
=&\: \frac{1}{\rho^2}\left[2\rho^2-1+a-\theta^2+\theta\,\frac{2(1-a-\rho)+a\rho}{\sqrt{(1-\rho)^2-a(1-\rho)}}\right]. \label{derivative2}
\end{align} 

Recall the definition 
$$\bar{\rho}=\bar{\rho}(\theta,a)=1-a/2-\sqrt{(a/2)^2+\theta^2},$$
which is obtained by choosing $\rho$ so as to kill the second term on the right-hand side of \eqref{f}.\footnote{$\rho=\bar\rho$ corresponds to the strategy that the initial particle is not moved to a linear distance in the interval $[0,\rho t]$; in this case, the only exponential probability cost comes from suppressing the branching over $[0,\rho t]$.} 
The minimization problem of $f$ satisfies the following properties. 

\begin{proposition}\label{prop5} For each $\theta,a\in\mathbb{R}$ satisfying $0<a<1-\theta^2<1$, let $f_{\theta,a}:(0,1-a]\to\mathbb{R}_+$ be the function defined in \eqref{f}.       
\begin{itemize}
	\item[(i)] The function $f_{\theta,a}$ is strictly convex, and has a unique minimizer on $(0,1-a)$. Denote this minimizer by $\hat{\rho}=\hat{\rho}(\theta,a)$. Then, $\hat{\rho}$ satisfies $0<\hat\rho<\bar\rho$.
	\item[(ii)] For fixed $\theta\in(0,1)$, both $\hat\rho$ and $f(\hat\rho)$ are strictly decreasing over $a\in[0,1-\theta^2)$.
	\item[(iii)] For fixed $\theta\in(0,1)$, 
	\begin{align}
&\quad \underset{a\rightarrow 0^+}{\lim}\hat{\rho}(a)=\hat\rho(0)=\frac{1-\theta}{\sqrt{2}}, \quad \quad
\underset{a\rightarrow 0^+}{\lim}f_a(\hat\rho(a))=f_0(\hat\rho(0))=2(\sqrt{2}-1)(1-\theta),\label{eq27} \\
&\underset{a\rightarrow (1-\theta^2)^-}{\lim}\hat{\rho}(a)=0, \quad \quad \quad \quad \quad
\underset{a\rightarrow (1-\theta^2)^-}{\lim}f_a(\hat\rho(a))=0.  \label{eq28}
	\end{align}
	\item[(iv)] For fixed $a\in(0,1)$, both $\hat\rho$ and $f(\hat\rho)$ are strictly decreasing over $\theta\in[0,\sqrt{1-a})$.
	\item[(v)] For fixed $a\in(0,1)$, 
	\begin{align}
&\quad\underset{\theta\rightarrow 0^+}{\lim}\hat{\rho}(\theta)=\hat\rho(0)=\sqrt{\frac{1-a}{2}},\quad \quad 
\underset{\theta\rightarrow 0^+}{\lim}f_\theta(\hat\rho(\theta))=f_0(\hat\rho(0))=2\sqrt{2(1-a)}-(2-a),\label{eq29}  \\
&\underset{\theta\rightarrow\sqrt{1-a}^-}{\lim}\hat{\rho}(\theta)=0, \quad \quad \quad \quad \quad \quad \:\:
\underset{\theta\rightarrow\sqrt{1-a}^-}{\lim}f_\theta(\hat\rho(\theta))=0.  \label{eq26}
	\end{align}
\end{itemize}
\end{proposition}

\noindent \textbf{Proof of (i)}. Note that $f'$ is continuous on $(0,1-a)$, and it is easy to check that when $0<a<1-\theta^2<1$,  
$$ \underset{\rho\rightarrow 0^+}{\lim}f'(\rho)=-\infty, \quad \underset{\rho\rightarrow(1-a)^-}{\lim}f'(\rho)=\infty.\footnote{When $\theta=0$ or $a=0$, $\lim_{\rho\rightarrow(1-a)^-}f'(\rho)<\infty$.} $$
It follows that $f'$ has a root in $(0,1-a)$. Also, one can check that the second derivative of $f$, which we denote by $f''$, is strictly positive on $(0,1-a)$. This means $f'$ is strictly increasing on $(0,1-a)$, implying that $f'$ has at most one root in $(0,1-a)$. Hence, $f'$ has exactly one root in $(0,1-a)$. Since $f''>0$, this root is a minimizer. Then, left continuity of $f$ at $\rho=1-a$ implies that $f$ has a unique minimizer, denoted by $\hat\rho$, on $(0,1-a]$. It is easy to check that $f'(\bar\rho)=1$, which, being positive, implies that $0<\hat\rho<\bar\rho<1-a$. This completes the proof of (i).

\bigskip

Before turning to the proof of Proposition~\ref{prop5} parts (ii)-(v), we first state and prove two lemmas, which will be used in the subsequent proofs.

\begin{lemma} \label{lemma2} Let $f_{\theta,a}:(0,1-a]\to\mathbb{R}_+$ be the function defined in \eqref{f} and $f'_{\theta,a}$ denote its derivative. 
\begin{itemize} 
\item[(i)] Fix $\theta\in(0,1)$. If $0<a_1<a_2<1-\theta^2$, then $f'_{a_2}(\hat\rho(a_1))>0$.
\item[(ii)] Fix $a\in(0,1)$. If $0<\theta_1<\theta_2<\sqrt{1-a}$, then $f'_{\theta_2}(\hat\rho(\theta_1))>0$.
\end{itemize}
\end{lemma}
\begin{proof}
Let $P=P_{\theta,a}$ be the sixth degree polynomial defined by the rule
\begin{align}
P(\rho)=&\left[4\rho^4-4(1+\theta^2-a)\rho^2+(1-\theta^2-a)^2\right]\left[\rho^2-(2-a)\rho+(1-a)\right]-(\theta a)^2\rho^2 \label{P} \\
=&\left[(1-\theta^2-a-2\rho^2)^2-8\theta^2\rho^2\right]\left[\rho^2-(2-a)\rho+(1-a)\right]-(\theta a)^2\rho^2. \label{P2}
\end{align} 
From \eqref{derivative2}, it can be shown that for $\rho\in(0,1-a]$ we have the following implications: 
\begin{align}
(a)\quad f'(\rho)=0 &\:\Rightarrow\: P(\rho)=0, \nonumber \\
(b)\quad f'(\rho)>0 &\:\Leftrightarrow\: P(\rho)<0. \nonumber
\end{align}
To prove (i), we first show that for any pair $(\theta,a)$ with $0<a<1-\theta^2<1$, we have $\hat\rho(\theta,a)<\sqrt{(1-\theta^2-a)/2}$. It is easy to check via \eqref{P} that $P((\sqrt{1-a}-\theta)/\sqrt{2})<0$, which implies by (b) that $f'((\sqrt{1-a}-\theta)/\sqrt{2})>0$. Since $f'(\hat\rho)=0$ and $f'$ is strictly increasing on $(0,1-a)$, this implies $\hat\rho<(\sqrt{1-a}-\theta)/\sqrt{2}$, which in turn implies $\hat\rho<\sqrt{(1-\theta^2-a)/2}$. 

Now fix $\theta$. Choose $\varepsilon_0>0$ small enough such that $a_1+\varepsilon_0<1-\theta^2$ and $1-\theta^2-(a_1+\varepsilon_0)-2(\hat\rho(a_1))^2>0$. (We have already shown above that $1-\theta^2-a_1-2(\hat\rho(a_1))^2>0$.) Set $a_2=a_1+\varepsilon_0$ and start with $f'_{a_1}(\hat\rho(a_1))=0$. Then, (a) implies that $P_{a_1}(\hat\rho(a_1))=0$. Since $1-\theta^2-a_2-2(\hat\rho(a_1))^2>0$ by the choice of $a_2$, \eqref{P2} implies that $P_{a_2}(\hat\rho(a_1))<0$, which by (b) implies that $f'_{a_2}(\hat\rho(a_1))>0$ for $a_2=a_1+\varepsilon_0$. Now since $f'$ is strictly increasing on $(0,1-a)$, we conclude that $f'_{a_2}(\hat\rho(a_1))>0$ for any $a_2$ with $a_1<a_2<1-\theta^2$.       

To prove (ii), fix $a$ and start with $f'_{\theta_1}(\hat\rho(\theta_1))=0$. By (a), this implies that $P_{\theta_1}(\hat\rho(\theta_1))=0$. By \eqref{P}, for $\theta_2>\theta_1$ it is clear that $P_{\theta_2}(\hat\rho(\theta_1))<0$, which by (b) implies that $f'_{\theta_2}(\hat\rho(\theta_1))>0$.  
\end{proof}

The following lemma says that when viewed as a function of $\theta$ and $a$, $\hat\rho$ is separately continuous.

\begin{lemma} \label{lemma4} Let $f_{\theta,a}:(0,1-a]\to\mathbb{R}_+$ be the function defined in \eqref{f} and $\hat\rho(\theta,a)$ denote its unique minimizer. 
\begin{itemize} 
\item[(i)] Fix $\theta\in(0,1)$ and let $(a_n)_{n\geq 1}$ be a sequence with all terms in $[0,1-\theta^2)$.\newline
If $a_n\to a_0$ and $a_0\in[0,1-\theta^2)$, then $\hat\rho(a_n)\to \hat\rho(a_0)$ as $n\to\infty$.
\item[(ii)] Fix $a\in(0,1)$ and let $(\theta_n)_{n\geq 1}$ be a sequence with all terms in $[0,\sqrt{1-a})$. \newline
If $\theta_n\to \theta_0$ and $\theta_0\in[0,\sqrt{1-a})$, then $\hat\rho(\theta_n)\to \hat\rho(\theta_0)$ as $n\to\infty$.
\end{itemize}
\end{lemma} 
\begin{proof} We prove (i). The proof of (ii) is similar. 
Fix $\theta$ and let $a_n\to a_0$ with $a_0\in[0,1-\theta^2)$. Since $\{\hat\rho(a_n)\}$ is a bounded sequence (all terms are in $[0,1]$), it has a convergent subsequence, say $\{\hat\rho(a_{n_k})\}$ with limit $\rho_0$. Let $\rho^*$ be the minimizer of $f_{a_0}$. Then, $f_{a_{n_k}}(\rho^*)\to f_{a_0}(\rho^*)$ since $a_{n_k}\to a_0$ and $f_a$ is continuous in $a$. Since $f_{a_{n_k}}(\rho^*)\geq f_{a_{n_k}}(\hat\rho(a_{n_k}))$, $a_{n_k}\to a_0$ and $\hat\rho(a_{n_k})\to\rho_0$, continuity of $f$ in both $a$ and $\rho$ implies that $f_{a_0}(\rho^*)\geq f_{a_0}(\rho_0)$. By uniqueness of minimizers, we conclude that $\rho_0=\rho^*$. This implies that $\hat\rho(a_n)\to \rho^*=\hat\rho(a_0)$.   
\end{proof}

\bigskip

\noindent \textbf{Proof of (ii)}. For fixed $\theta$, let $a_1$ and $a_2$ be in $[0,1-\theta^2)$ with $a_1<a_2$. Observe that \eqref{f} implies $f_{a_1}(\rho)>f_{a_2}(\rho)$ for each $\rho\in(0,1-a_2]$. Then,
$$ f_{a_1}(\hat\rho(a_1))>f_{a_2}(\hat\rho(a_1))\geq f_{a_2}(\hat\rho(a_2)) \quad \text{when} \quad \hat\rho(a_1)\in(0,1-a_2],$$
and
$$ f_{a_1}(\hat\rho(a_1))\geq \hat\rho(a_1)>1-a_2\geq f_{a_2}(\hat\rho(a_2)) \quad \text{when} \quad \hat\rho(a_1)\in(1-a_2,1-a_1],$$
where the last inequality follows since $f_{a_2}(\bar\rho(a_2))=\bar\rho(a_2)<1-a_2$ and $\hat\rho(a_2)$ is the minimizer for $f_{a_2}$. This implies that $f_{a}(\hat\rho(a))$ is strictly decreasing in $a$. 

To show that $\hat\rho(a)$ is strictly decreasing in $a\in[0,1-\theta^2)$, let $a_1<a_2$, and use that $f'_{a_2}(\hat\rho(a_1))>0$ (Lemma~\ref{lemma2}(i)), which implies that $\hat\rho(a_2)<\hat\rho(a_1)$ since $f'$ is strictly increasing. This completes the proof of (ii). 

\bigskip

\noindent \textbf{Proof of (iii)}. It follows directly from Lemma~\ref{lemma4}(i) that $\underset{a\rightarrow 0^+}{\lim}\hat{\rho}(a)=\hat\rho(0)$. Also, if $a_n\to a$ then $f_{a_n}(\hat\rho(a_n))\to f_a(\hat\rho(a))$ since $a_n\to a$, $\hat\rho(a_n)\to\hat\rho(a)$ by Lemma~\ref{lemma4}(i), and $f$ is continuous in both $a$ and $\rho$. It follows that $\underset{a\rightarrow 0^+}{\lim}f_a(\hat\rho(a))=f_0(\hat\rho(0))$. 

We now prove \eqref{eq28}. Fix $\theta$ and let $a=1-\theta^2-\varepsilon$, where $\varepsilon>0$ is small. Write the second term on the right-hand side of \eqref{f} as 
\begin{align}
\frac{\left[\sqrt{(1-\rho)^2-a(1-\rho)}-\theta\right]^2}{\rho}=&\:\frac{\left[(1-\rho)^2-a(1-\rho)-\theta^2\right]^2}{\rho\left[\sqrt{(1-\rho)^2-a(1-\rho)}+\theta\right]^2} \nonumber \\ 
=&\:\frac{\rho(\rho-1-\theta^2+\varepsilon(1-\rho)/\rho)^2}{\left[\sqrt{(1-\rho)^2-(1-\theta^2-\varepsilon)(1-\rho)}+\theta\right]^2} \nonumber \\ 
\leq&\:\varepsilon\: \theta^4/\theta^2\leq \varepsilon,
\end{align}
where we have used $a=1-\theta^2-\varepsilon$ in passing to the second line, and we set $\rho=\varepsilon$ in passing to the last line. This shows that for $a=1-\theta^2-\varepsilon$, $\underset{\rho\in(0,1-a]}{\inf}f_a(\rho)\leq f_a(\varepsilon)\leq 2\varepsilon$, which proves that $\underset{a\rightarrow (1-\theta^2)^-}{\lim}f_a(\hat\rho(a))=0$. Note that for sufficiently small $\varepsilon$, we may take $\rho=\varepsilon$ since $\theta>0$, and we have $1-a=1-(1-\theta^2-\varepsilon)=\theta^2-\varepsilon>\varepsilon$ so that $\rho=\varepsilon$ falls inside $(0,1-a]$. Since both terms in the formula for $f(\rho)$ are positive (see \eqref{f}), this in turn implies that $\underset{a\rightarrow (1-\theta^2)^-}{\lim}\hat{\rho}(a)=0$. This completes the proof of (iii). 

\bigskip

\noindent \textbf{Proof of (iv)}. For fixed $a$, let $\theta_1$ and $\theta_2$ be in $[0,\sqrt{1-a})$ with $\theta_1<\theta_2$. Observe that \eqref{f} implies $f_{\theta_1}(\rho)>f_{\theta_2}(\rho)$ for each $\rho\in(0,1-a]$. 
Then,
$$ f_{\theta_1}(\hat\rho(\theta_1))>f_{\theta_2}(\hat\rho(\theta_1))\geq f_{\theta_2}(\hat\rho(\theta_2)),$$
which implies that $f_{\theta}(\hat\rho(\theta))$ is strictly decreasing in $\theta$. To show that $\hat\rho(\theta)$ is strictly decreasing in $\theta\in[0,\sqrt{1-a})$, let $\theta_1<\theta_2$, and use that $f'_{\theta_2}(\hat\rho(\theta_1))>0$ (Lemma~\ref{lemma2}(ii)), which gives $\hat\rho(\theta_2)<\hat\rho(\theta_1)$ since $f'$ is strictly increasing. This completes the proof of (iv).

\bigskip

\noindent\textbf{Proof of (v)}. It follows directly from Lemma~\ref{lemma4}(ii) that $\underset{\theta\rightarrow 0^+}{\lim}\hat{\rho}(\theta)=\hat\rho(0)$. Also, if $\theta_n\to \theta$ then $f_{\theta_n}(\hat\rho(\theta_n))\to f_\theta(\hat\rho(\theta))$ since $\theta_n\to \theta$, $\hat\rho(\theta_n)\to\hat\rho(\theta)$ by Lemma~\ref{lemma4}(ii), and $f$ is continuous in both $\theta$ and $\rho$. It follows that $\underset{\theta\rightarrow 0^+}{\lim}f_\theta(\hat\rho(\theta))=f_0(\hat\rho(0))$. 

We now prove \eqref{eq28}. Fix $a$ and let $\theta=\sqrt{1-a-\varepsilon}$, where $\varepsilon>0$ is small. Write the second term on the right-hand side of \eqref{f} as 
\begin{align}
\frac{\left[\sqrt{(1-\rho)^2-a(1-\rho)}-\theta\right]^2}{\rho}=&\:\frac{\left[(1-\rho)^2-a(1-\rho)-\theta^2\right]^2}{\rho\left[\sqrt{(1-\rho)^2-a(1-\rho)}+\theta\right]^2} \nonumber \\ 
=&\:\frac{\rho(\rho-2+a+\varepsilon/\rho)^2}{\left[\sqrt{(1-\rho)^2-a(1-\rho)}+\sqrt{1-a-\varepsilon}\right]^2} \nonumber \\ 
\leq&\: \frac{\varepsilon(\varepsilon+a-1)^2}{1-a-\varepsilon}\leq \varepsilon(\varepsilon+a-1)\leq \varepsilon,
\end{align}
where we have used $\theta=\sqrt{1-a-\varepsilon}$ in passing to the second line, and we set $\rho=\varepsilon$ for $\varepsilon<1-a$ in passing to the last line. This shows that for $\theta=\sqrt{1-a-\varepsilon}$, $\underset{\rho\in(0,1-a]}{\inf}f_\theta(\rho)\leq f_\theta(\varepsilon)\leq 2\varepsilon$, which proves that $\underset{\theta\rightarrow\sqrt{1-a}^-}{\lim}f_\theta(\hat\rho(\theta))=0$. Since both terms in the formula for $f(\rho)$ are positive,  this in turn implies that $\underset{\theta\rightarrow\sqrt{1-a}^-}{\lim}\hat{\rho}(\theta)=0$, completing the proof of (v). \qed

\section{Proof of Theorem~\ref{theorem2}} \label{section6}
In this section we prove Theorem~\ref{theorem2}. The method of proof is similar to that of Theorem~\ref{theorem1}. 

\subsection{Proof of the lower bound}
The lower bound arises from the following strategy which gives a mass of less than $\exp(\beta a t)$ to $\widehat{B}_t^c$ at time $t$, where $\widehat{B}_t:=B(0,x_t)$ with $x_t=\theta\sqrt{2\beta}t$ and $0< \theta<1$.

For $0<\theta<1$ and $0\leq a<1-\theta^2$, as before, let 
$$ \bar{\rho}:=1-a/2-\sqrt{(a/2)^2+\theta^2},$$
and for $0<\delta_1<1-\bar\rho$, split the time interval $[0,t]$ into two pieces at $(\bar{\rho}+\delta_1)t$. In $[0,(\bar{\rho}+\delta_1)t]$, suppress the branching completely and for $0<\delta_2<\theta$, confine the single particle to $B(0,\delta_2\sqrt{2\beta}t)$. Suppressing the branching costs a probability of $\exp[-\beta (\bar{\rho}+\delta_1)t]$, whereas confining the particle to $B(0,\delta_2\sqrt{2\beta}t)$ comes for `free', i.e., costs $\exp[o(t)]$, since a Brownian particle typically moves a distance of order $\sqrt{t}$ over linear time. 

Next, in the remaining interval $[(\bar{\rho}+\delta_1)t,t]$, let the BBM behave `normally,' and send typically many particles to $\widehat{B}_t^c$ at time $t$. To be precise, denote the position of the single particle at time $(\bar{\rho}+\delta_1)t$ by $X_t$. Then, uniformly over $X_t\in B(0,\delta_2\sqrt{2\beta}t)$, \eqref{biggins} implies that for any $\varepsilon>0$, for all large $t$,
\begin{equation} P_{X_t}\left(Z_{(1-(\bar{\rho}+\delta_1))t}(\widehat{B}_t^c)\leq \exp\left[\left(1-(\bar{\rho}+\delta_1)-\frac{(\theta-\delta_2)^2}{1-(\bar{\rho}+\delta_1)}\right)\beta t+\varepsilon\beta t\right]\right)>1-\varepsilon. 
\end{equation}
Choose $\varepsilon$ and $\delta_2$ small enough to satisfy
\begin{equation}
\varepsilon+\left[1-(\bar{\rho}+\delta_1)-\frac{(\theta-\delta_2)^2}{1-(\bar{\rho}+\delta_1)}\right]< a. \label{alev}
\end{equation}
Indeed, it is possible to choose $\varepsilon$ and $\delta_2$ satisfying \eqref{alev}. Using the definition of $\bar\rho$, it is not hard to show that the left-hand side of \eqref{alev} can be bounded from above by $\varepsilon+a-\delta_1+2\delta_2/\theta$, which is less than $a$ if we choose $\delta_2=\theta\delta_1/4$ and $\varepsilon=\delta_1/4$.

Applying the Markov property at $(\bar\rho+\delta_1)t$, and combining the costs of the `partial' strategies over $[0,(\bar\rho+\delta_1)t]$ and $[(\bar\rho+\delta_1)t,t]$, we obtain
$$\underset{t\rightarrow\infty}{\liminf}\:\frac{1}{t}\log P\left(Z_t(\widehat{B}_t^c)< e^{\beta at}\right)\geq -\beta (\bar{\rho}+\delta_1)=-\beta[1-a/2-\sqrt{\theta^2+(a/2)^2}+\delta_1] .$$
Let $\delta_1\rightarrow 0$ to obtain the lower bound.

\subsection{Proof of the upper bound}
Let $0<\delta<\bar{\rho}$. Estimate
\begin{equation} P\left(Z_t(\widehat{B}_t^c)< e^{\beta a t}\right)\leq 
P\left(N_{(\bar{\rho}-\delta)t}\leq \lfloor t \rfloor\right)+P\left(Z_t(\widehat{B}_t^c)< e^{\beta a t}\mid N_{(\bar{\rho}-\delta)t}>\lfloor t \rfloor\right). \label{eq30}
\end{equation}
We will show that the second term on the right-hand side of \eqref{eq30} is SES in $t$. Consider a particle that is present at time $(\bar{\rho}-\delta)t$. Regardless of its position at that time, it is at most at a distance of $x_t$ away from the boundary of $\widehat{B}_t$. Recall that $x_t=\theta\sqrt{2\beta}t$, which is small enough for the sub-BBM emanating from the particle at time $(\bar{\rho}-\delta)t$ to contribute at least $\exp(\beta a t)$ particles to $\widehat{B}_t^c$ at time $t$. More precisely, the remaining time is $(1-\bar\rho+\delta)t$, and since
\begin{equation} a<(1-\bar\rho+\delta)-\frac{\theta^2}{1-\bar\rho+\delta}, \label{eq300}
\end{equation}
Lemma~\ref{lemma3} implies that each particle present at time $(\bar{\rho}-\delta)t$ initiates a sub-BBM which, with overwhelming probability, contributes at least $e^{\beta a t}$ particles to $\widehat{B}_t^c$ at time $t$, that is, the probability of the complement event is at most $e^{-ct}$ for some $c$ for all large $t$. (It is easy to see \eqref{eq300} since $a=(1-\bar\rho)-\frac{\theta^2}{1-\bar\rho}$ and $\delta>0$.) Conditional on the event $\{ N_{(\bar{\rho}-\delta)t}>\lfloor t \rfloor \}$, there are at least $\lfloor t \rfloor$ particles at time $(\bar{\rho}-\delta)t$. Therefore, by independence of sub-BBMs emanating from these particles at that time, for all large $t$,
\begin{equation}
P\left(Z_t(\widehat{B}_t^c)< e^{\beta a t}\mid N_{(\bar{\rho}-\delta)t}>\lfloor t \rfloor\right)\leq (e^{-ct})^{\lfloor t \rfloor}, \label{eq31}
\end{equation}
which is SES in $t$. Next, we use \eqref{prop1} to bound the first term on the right-hand side of \eqref{eq30} from above as
\begin{equation} P\left(N_{(\bar{\rho}-\delta)t}\leq \lfloor t \rfloor\right)\leq e^{-\beta(\bar{\rho}-\delta)t+o(t)}.  \label{eq32}
\end{equation}
Using \eqref{eq30}, \eqref{eq31}, and \eqref{eq32}, we obtain
$$\underset{t\rightarrow\infty}{\limsup}\:\frac{1}{t}\log P\left(Z_t(\widehat{B}_t^c)< e^{\beta a t}\right)\leq -\beta(\bar{\rho}-\delta)=-\beta[1-a/2-\sqrt{\theta^2+(a/2)^2}-\delta].$$
Let $\delta\rightarrow 0$ to complete the proof.

\bibliographystyle{plain}

\begin{thebibliography}{1}

\bibitem{A2017}
E. A\"{i}dekon, Y. Hu and Z. Shi.
\newblock Large deviations for level sets of branching Brownian motion and Gaussian free fields.
\newblock arXiv:1710.00348

\bibitem{AH1976}
S. Asmussen and H. Hering.
\newblock Strong limit theorems for general supercritical branching processes with applications to branching diffusions.
\newblock {\em Z. Wahrsch. Verw. Geb.} \textbf{36} (1976) 195 -- 212.

\bibitem{AH1983}
S. Asmussen and H. Hering.
\newblock {\em Branching Processes}. Birkh\"auser, Basel, 1983.

\bibitem{AN1972}
K. Athreya and P. Ney.
\newblock {\em Branching Processes}. Springer-Verlag, Berlin, 1972.

\bibitem{B1992}
J. D. Biggins.
\newblock Uniform convergence in the branching random walk.
\newblock {\em Annals of Probability} \textbf{20} (1992) 137 -- 151.

\bibitem{B1978}
M. Bramson.
\newblock Maximal displacement of branching Brownian motion.
\newblock {\em Communications on Pure and Applied Mathematics} \textbf{31} (5) (1978) 531 -- 581.

\bibitem{CR1988}
B. Chauvin and A. Rouault.
\newblock KPP equation and supercritical branching brownian motion in the subcritical speed area. Application to spatial trees.
\newblock {\em Probability Theory and Related Fields} \textbf{80} (1988) 299 -- 314.

\bibitem{CS2007}
Z.-Q. Chen and Y. Shiozawa.
\newblock Limit theorems for branching Markov processes.
\newblock {\em Journal of Functional Analysis} \textbf{250} (2007) 374 -- 399.

\bibitem{E2003}
J. Engl\"{a}nder and F. den Hollander.
\newblock Survival asymptotics for branching Brownian motion in a Poissonian trap field.
\newblock {\em Markov Processes and Related Fields} \textbf{9} (2003) 363 -- 389.

\bibitem{E2004}
J. Engl\"{a}nder.
\newblock Large deviations for the growth rate of the support of supercritical super-Brownian motion.
\newblock {\em Statistics and Probability Letters} \textbf{66} (4) (2004) 449 -- 456.

\bibitem{EK2003}
J. Engl\"{a}nder and A. E. Kyprianou.
\newblock Local extinction versus local exponential growth for spatial branching processes.
\newblock {\em Annals of Probability} \textbf{32} (2003) 78 -- 99.

\bibitem{E2007}
J. Engl\"{a}nder.
\newblock Branching diffusions, superdiffusions and random media.
\newblock {\em Probability Surveys} \textbf{4} (2007) 303 -- 364.

\bibitem{EHK2010}
J. Engl\"{a}nder, S. C. Harris and A. E. Kyprianou.
\newblock Strong law of large numbers for branching diffusions.
\newblock {\em Annales de l'Institut Henri Poincar\'e, Probabilit\'es et Statistiques} \textbf{46} (1) (2010) 279 -- 298.

\bibitem{KT1975}
S. Karlin and M. Taylor.
\newblock {\em A First Course in Stochastic Processes}. Academic Press, New York, 1975.

\bibitem{K2005}
A. E. Kyprianou.
\newblock Asymptotic radial speed of the support of supercritical branching Brownian motion and super-Brownian motion in $R^d$.
\newblock{\em Markov Processes and Related Fields} \textbf{11} (2005) 145 -- 156.

\bibitem{MK1975}
H. P. McKean.
\newblock Application of Brownian motion to the equation of Kolmogorov-Petrovskii-Piskunov.
\newblock {\em Communications in Pure and Applied Mathematics} \textbf{28} (1975) 323 -- 331.

\bibitem{OCE2017}
M. \"{O}z, M. \c{C}a\u{g}lar and J. Engl\"{a}nder.
\newblock Conditional speed of branching Brownian motion, skeleton decomposition and application to random obstacles. 
\newblock{\em Annales de l'Institut Henri Poincar\'e, Probabilit\'es et Statistiques} \textbf{53} (2) (2017) 842 -- 864.

\bibitem{OE2018}
M. \"{O}z and J. Engl\"{a}nder.
\newblock Optimal survival strategy for branching Brownian motion in a Poissonian trap field. To appear in {\em Annales de l'Institut Henri Poincar\'e, Probabilit\'es et Statistiques}. 

\bibitem{S2017}
Y. Shiozawa.
\newblock Spread rate of branching Brownian motions.
\newblock {\em Acta Applicandae Mathematicae} \textbf{155} (1) (2018) 113 -- 150.

\bibitem{W1967}
S. Watanabe.
\newblock Limit theorems for a class of branching processes, in: J. Chover (Ed.)
\newblock {\em Markov Processes and Potential Theory}. Wiley, New York, (1967) 205 -- 232.
 

\end{thebibliography}

\end{document}